\documentclass[11pt,a4paper]{article}
\usepackage{amsmath,amsthm,amssymb}
\usepackage{graphicx}
\usepackage{subfigure}
\usepackage{multirow}
\usepackage{epstopdf}
\usepackage{verbatim}
\usepackage{color}
\newtheorem{theorem}{Theorem}[section]

\newtheorem{lemma}{Lemma}[section]

\newtheorem{remark}{Remark}[section]
\newcommand{\se}{\setcounter{equation}{0}}

\newcommand{\cop}{ \partial_t^{\alpha}}

\newcommand{\norm}[1]{{\left\vert\kern-0.25ex\left\vert\kern-0.25ex\left\vert #1 
    \right\vert\kern-0.25ex\right\vert\kern-0.25ex\right\vert}}

\title{Numerical solution of a time-fractional nonlinear Rayleigh-Stokes problem}
\author{Mariam Al-Maskari\thanks{Email: m.maskari@student.squ.edu.om} \quad and\quad 
Samir Karaa\thanks{Email: skaraa@squ.edu.om.  This work is supported bu Sultan Qaboos University under grant  IG/SCI/MATH/20/04.} \\
\\
Department of Mathematics, Sultan Qaboos University\\
Al-Khod 123, Muscat, Oman}
\begin{document}
\date{}
\maketitle
\begin{abstract}
We study a  semilinear fractional-in-time  Rayleigh-Stokes problem for a generalized second-grade fluid with a Lipschitz continuous nonlinear source term and  initial data $u_0\in\dot{H}^\nu(\Omega)$, $\nu\in[0,2]$.  We discuss  stability of solutions and provide regularity results. Two  spatially semidiscrete schemes are analyzed based on standard Galerkin and lumped mass finite element methods, respectively.   Further, a fully discrete scheme  is obtained by applying a convolution quadrature in time  generated by the backward Euler method, and optimal error estimates  are derived for  smooth and nonsmooth initial data. Finally,  numerical examples are provided  to illustrate the theoretical results.
\end{abstract}



\section{Introduction}
We consider  a  semilinear fractional-order Rayleigh-Stokes problem  for a generalized second-grade fluid. Let $\Omega\subset \mathbb{R}^d\, (d=1,2,3)$ be a bounded convex polygonal domain with its boundary $\partial \Omega$, and $T>0$. The  mathematical model is given by
\begin{subequations}\label{main}
\begin{alignat}{2}\label{a1}
& \partial_t  u(x,t) -(1+\gamma \partial_t^{\alpha})\Delta u(x,t)=f(u) &&\quad\mbox{ in }\Omega\times (0,T],
\\  \label{a2}
&u(x,t)= 0 &&\quad\mbox{ on }\partial\Omega\times (0,T],
\\   \label{a3}
&u(x,0)=u_0(x) &&\quad\mbox{ in }\Omega,
\end{alignat}
\end{subequations}
where $\gamma>0$ is a fixed constant, $u_0$ is a given initial data, $\partial_t=\partial /\partial t$ and 
$\partial_t^{\alpha}$ is the Riemann-Liouville  fractional derivative in time  \textcolor{black}{with} $\alpha\in(0,1)$ defined by
\begin{equation} \label{Ba}
\partial_t^{\alpha} \varphi(t)=\frac{d}{dt}\int_0^t\omega_{1-\alpha}(t-s)\varphi(s)\,ds\quad\text{with} \quad
\omega_{\alpha}(t):=\frac{t^{\alpha-1}}{\Gamma(\alpha)}.
\end{equation}
In \eqref{a1},  $f:\mathbb{R}\to\mathbb{R}$ is a smooth function satisfying the Lipschitz condition 
\begin{equation} \label{Lip}
|f(t)-f(s)|\leq L|t-s|\quad \forall t,s\in \mathbb{R},
\end{equation} 
for some constant $L>0$.


The aim of this work is to study some aspects of the numerical solution of the semilinear problem \eqref{main}. The linear case has been considered by several authors.
For instance, in \cite{stokes1} and \cite{stokes2}, implicit and explicit finite difference schemes have been proposed. A Fourier analysis was employed to investigate stability and convergence. In \cite{23}, a numerical scheme was derived and analyzed by transforming the problem into an integral equation. In \cite{stokes4}, a numerical scheme was investigated  
using the reproducing kernel technique. 
In \cite{Zaky}, Zaky applied the Legendre-tau method to problem \eqref{main} and  discussed related  convergence rates.
 The convergence analysis in all these studies assumes that the exact solution  
is sufficiently regular, including at $t=0$, which is not practically the case. 
In \cite{EJLZ2016},   Jin et al.  investigated a piecewise linear finite element method (FEM) in space and a convolution quadrature in time, and  obtained optimal error estimates with respect to the solution smoothness, expressed through the initial data $u_0$. Most recently, a similar analysis was presented in \cite{MK-2018} for a time-fractional Oldroyd-B fluid problem.

The numerical approximation of nonlinear time-fractional models has recently attracted the attention of many researchers.  In particular, the time-fractional subdiffusion model
\begin{equation}\label{uu}
^C\partial_t^{\alpha} u(x,t) -\Delta u(x,t)=f(u)
\end{equation}
has \textcolor{black}{been} given a special attention. Here,  $^C\partial_t^{\alpha}$ denotes the Caputo fractional derivative in time of order $\alpha$.  
 In \cite{LWZ-2017}, for instance, a linearized $L^1$-Galerkin FEM was proposed for solving a nonlinear time-fractional Schr\"odinger equation. Based on a temporal-spatial error splitting argument and a new discrete fractional Gronwall-type inequality, optimal error estimates of the numerical schemes are obtained without 
restrictions on the time step size. In \cite{LLSWZ-2018},  $L^1$-type schemes have been analyzed for approximating the solution of \eqref{uu}, and related  error estimates have been derived.  The estimates are obtained under high regularity assumptions on the exact solution. 
 In \cite{JLZ-2018}, the numerical solution of \eqref{uu} was investigated under the assumption that the nonlinear function $f$ is globally Lipschitz continuous  and the  initial data $u_0\in H^2(\Omega)\cap H^1_0(\Omega)$. These results have been extended in \cite{MK-2019} to problems with nonsmooth initial data. Recently,  a numerical study with a more general condition on nonlinearity was presented in \cite{MK-2020}.

In this paper, we first investigate  a lumped mass FE semidiscrete scheme in space for solving \eqref{main}.   \textcolor{black}{ Compared with the standard piecewise linear FEM \cite{MK-2019,EJLZ2016},
 the lumped mass FEM has the advantage that 
when representing the discrete solution in the nodal basis functions, it produces  a diagonal mass matrix which enhances the computation procedure.} 
 Our aim  is to  derive optimal error estimates for solutions with smooth and nonsmooth initial data. The analysis will be based on a semi-group type approach.
The FE solution will serve as an intermediate solution to establish error estimates for
 the lumped mass FEM. This technique was used for instance in \cite{CLT-2012,CLT-2013} and 
 \cite{MK-2018-b}.
Our second objective is to investigate a  time-stepping scheme using a first-order convolution quadrature in time. Pointwise-in-time optimal  error estimates 
are then derived. The main technical tool relies on the use of the discrete propagator (discrete evolution operator) associated with the numerical method, see \cite{Lubich-2006}.

The paper   is organized as follows. In Section 2, we represent the solution of \eqref{main} in an integral form and obtain regularity results. In Section 3, we derive  error estimates for the standard Galerkin FEM. A  convolution quadrature time discretization method is analyzed in Section 4, and  related error estimates  are established. In Section 5, we investigate a fully discrete scheme obtained by the lumped mass FEM combined with the convolution quadrature in time.  Finally, we provide some numerical examples to confirm our theoretical results.

Throughout the paper,  $c$  denotes a generic constant which may change at each occurrence but it is always  independent  of discretization parameters; mesh size $h$ and time step size $\tau$. We shall also use the notation $u'$  denoting $\partial u/\partial t$.
%
\section{Continuous problem} \label{sec:notation}
\se
This section is devoted to the analysis of the continuous problem \eqref{main}. Based on an integral    representation of its solution, we prove regularity results, which will play a key role in the error analysis. We begin by introducing some notations.
 For  $r\geq 0$, we denote by  $\dot H^r(\Omega)\subset L^2(\Omega)$ the Hilbert space induced by the norm 
$ \|v\|_{\dot H^r(\Omega)}^2=\sum_{j=1}^\infty \lambda_j^r (v,\phi_j)^2$, where  $\{(\lambda_j,\phi_j)\}_{j=1}^\infty$ are the Dirichlet eigenpairs  of  $A:=-\Delta$ on $\Omega$ with $\{\phi_j\}_{j=1}^\infty$ being an orthonormal basis in $L^2(\Omega)$. Thus, 
 $\|v\|_{\dot H^0(\Omega)}=\|v\|$ is the norm in $L^2(\Omega)$, 
$\|v\|_{\dot H^1(\Omega)}$ is the norm in $H_0^1(\Omega)$, and $\|v\|_{\dot H^2(\Omega)}=\|A v\|$ is the equivalent norm in $H^2(\Omega)\cap H^1_0(\Omega)$ \cite{thomee1997}.

For a given $\theta\in (\pi/2,\pi)$, we define the sector $
\Sigma_{\theta}=\{z\in \mathbb{C}, \,z\neq 0,\, |\arg z|< \theta\}
$. Since $A$ is selfadjoint and positive definite, the operator
$(z^\alpha I+A)^{-1}:L^2(\Omega)\to L^2(\Omega)$ satisfies the bound 
\begin{equation}\label{res1}
\|(z^\alpha I+A)^{-1}\|\leq M |z|^{-\alpha} \quad \forall z\in \Sigma_\theta,
\end{equation}
where $M$ depends on $\theta$.

Let $\hat{u}(x,z)$ denote the the Laplace transform of $u(x,t)$. Set  $w(t)=f(u(t))$. Then, by taking Laplace transforms in \eqref{a1}, we obtain
$$
z \hat{u} - u_0+A \hat{u}+\gamma z^{\alpha}A \hat{u}=\hat{w}(z).
$$
Hence,
$$
\hat{u}=\frac{g(z)}{z}\left(g(z)I+A\right)^{-1} \left( u_0+\hat {w}(z)\right),
$$
where $g(z)=\dfrac{z}{1+\gamma z^\alpha}$.
By means of the inverse Laplace transform, we have
\begin{equation}\label{form-1}
u(t)=E(t)u_0+\int_0^t  E(t-s)f(u(s))\,ds,\quad t>0,
\end{equation}
\textcolor{black}{with} the operator $E(t):L^2(\Omega)\to L^2(\Omega)$ being defined by
$$
 E(t) = \frac{1}{2\pi i}\int_{\Gamma_{\theta,\delta}}e^{zt}\frac{g(z)}{z}\left(g(z)I+A\right)^{-1} \,dz,
$$
\textcolor{black}{where, for fixed $\delta>0$, 
$\Gamma_{\theta,\delta}:=\{\rho e^{\pm i\theta}:\; \rho\geq\delta\}\cup \{\delta e^{i \psi}:\; |\psi|\leq\theta\} $
is oriented with an increasing imaginary part.}

The following estimates hold, see \cite[Theorem 2.1]{EJLZ2016}.
\begin{lemma}\label{LL} The operator $E(t)$ satisfies 

$$ \| \partial_t^m {E}(t)v\|_{\dot{H}^p(\Omega)}\leq ct^{-m-(1-\alpha)(p-q)/2}  \|v\|_{\dot{H}^q(\Omega)},$$
where  $m=0$ and $0\leq q\leq p\leq 2$ or  $m>0$ and $0\leq p,\, q \, \leq 2$.
\end{lemma}
In the sequel, we shall use  the following generalization of Gr\"onwall's inequality  \cite{Amann}.
\begin{lemma}\label{Gronwall}
Let $T > 0$, $0 \leq \alpha, \beta < 1$ and $A, B \geq 0$. Then there is a positive
constant $C = C(T,B,\alpha, \beta)$ such that
$$y(t)\leq At^{-\alpha} + B \int_0^t  (t- s)^{- \beta}y(s)ds,\quad\ 0 < t \leq T,$$
implies
$$y(t) \leq CAt^{-\alpha},\quad\ 0 < t \leq T.$$
\end{lemma}
Note that, by  the Lipschitz continuity of $f$, 
\begin{equation*}\label{f(u)}
\|f(u)\|\leq \|f(u)-f(0)\|+ \|f(0)\|\leq L\|u\|+ \|f(0)\|.
\end{equation*}
Using \eqref{form-1} and  Lemma \ref{LL}, we then get
\begin{eqnarray*}
\|u(t)\| & \leq & c\|u_0\| +c\int_0^t \|f(u(s))\|\,ds \\
& \leq & c\|u_0\| + ct\|f(0)\|+cL\int_0^t\|u(s)\|\,ds.
\end{eqnarray*}
\textcolor{black}{ By Lemma \ref{Gronwall}, we obtain the stability result 
\begin{eqnarray*}
\|u(t)\| & \leq & c(\|u_0\| + t\|f(0)\|).
\end{eqnarray*}}
Further properties of the solution $u$ are given below.
\begin{theorem}\label{T-1}  
Assume $u_0\in \dot{H}^\nu(\Omega)$, $\nu\in [0,2]$. Then problem  \eqref{main}  has a unique   solution  $u$ satisfying
\begin{equation}\label{regularity-1a}
 u\in C([0,T];\dot{H}^\nu(\Omega))\cap C((0,T];\dot{H}^2(\Omega)).
\end{equation}
Furthermore, 
\begin{equation}\label{regularity-1b}
  \| u(t) \|_{\dot{H}^p(\Omega)} \leq ct^{(\alpha-1) (p-\nu)/2}, \quad 0\leq \nu\leq p\leq 2,
\end{equation}
and 
\begin{equation}\label{regularity-2}
  \|u'(t) \|_{\dot{H}^p(\Omega)} \leq ct^{(\alpha-1) (p-\nu)/2-1}, \quad p\in [0,1].
\end{equation}
The constant $c$  may depend on $T$.
\end{theorem}
\begin{proof}  

For $\nu\in (0,2]$, the proof follows the same lines as that  of Theorem 3.1 in \cite{MK-2019}. The latter also  covers the estimate \eqref{regularity-1b} when $\nu=0$, see Step 3 in that proof. Thus, we shall only prove \eqref{regularity-2} for $\nu=0$. To do so, we differentiate both sides of \eqref{form-1} with respect to $t$ so that
\begin{equation}\label{derv}
\begin{split}
u'(t) =E'(t)u_0+ E(t)f(u_0)+\int_0^t  E(t-s)f'(u(s))u'(s)\,ds.
\end{split}
\end{equation} 
Multiplying by $t$, we have
$$tu'(t)=tE'(t)u_0+ tE(t)f(u_0)+\int_0^t s E(t-s)f'(u(s))u'(s)\,ds +\int_0^t (t-s) E(t-s)f'(u(s))u'(s)\,ds.$$
Following \cite[Lemma 5.2]{Mclean2010} and integrating by parts  the last term on the right hand side, we get 
$$\int_0^t (t-s) E(t-s)f'(u(s))u'(s)\,ds= -tE(t)f(u_0)+\int_0^t  E(t-s)f(u(s))\,ds+\int_0^t (t-s) E'(t-s)f(u(s))\,ds.  $$
Hence,
$$tu'(t)=tE'(t)u_0 +\int_0^t s E(t-s)f'(u(s))u' (s)\,ds+\int_0^t  E(t-s)f(u(s))\,ds+\int_0^t (t-s) E'(t-s)f(u(s))\,ds.$$
Using Lemma \ref{LL}, we thus deduce that 
$$\Vert tu'(t)\Vert \leq c +c \int_0^t \Vert s\, u'(s) \Vert\,ds, $$
which, by Lemma \ref{Gronwall}, implies that $\Vert tu'(t)\Vert \leq c.$
The $H^1(\Omega)$-estimate  $\| \nabla u'(t)\| \leq c t^{(\alpha-1)(1-\nu)/2-1}$ is derived in a similar manner. The desired estimate \eqref{regularity-2} follows then by interpolation, which  completes the proof.
\end{proof}

\section{Semidiscrete FE scheme} \label{sec:FE}
\se

Let $\mathcal{T}_h$ be a shape  regular and quasi-uniform triangulation of the domain $\bar\Omega$ into triangles  $K,$
and let $h=\max_{K\in \mathcal{T}_h}h_{K},$ where $h_{K}$ denotes the diameter  of $K.$  
The approximate solution $u_h$ of the Galerkin FEM will be sought in the FE space $V_h$ of continuous piecewise linear functions over the triangulation $\mathcal{T}_h$
$$V_h=\{v_h\in C^0(\overline {\Omega})\;:\;v_h|_{K}\;\mbox{is linear for all}~ K\in \mathcal{T}_h\; \mbox{and} \; v_h|_{\partial \Omega}=0\}.$$
The semidiscrete Galerkin FEM for problem (\ref{main}) now reads: find $u_h(t) \in V_h$ such that
\begin{equation} \label{semi-1}
(\partial_t u_{h} ,\chi)+  a( u_h,\chi)+ \gamma a(\partial_t^{\alpha} u_h,\chi)=  (f(u_h),\chi)\quad
\forall \chi\in V_h,\quad t\in (0,T], \quad u_h(0)=P_h u_0,
\end{equation}
where $(\cdot,\cdot)$ is the inner product in $L^2(\Omega)$, 
 $a(v,w):= (\nabla v, \nabla w)$ 
 and  
$P_h:L^2(\Omega)\rightarrow V_h$ is the orthogonal $L^2(\Omega)$-projection.
Upon introducing the discrete  operator $A_h:V_h\rightarrow V_h$ defined by
\begin{equation*} 
(A_h\psi,\chi)=(\nabla \psi,\nabla \chi)  \quad \forall \psi,\chi\in V_h,
\end{equation*}
the spatially discrete problem (\ref{semi-1}) is  equivalent to
\begin{equation} \label{semi-2}
 \partial_t u_{h}(t)+ ( 1+ \gamma \partial_t^{\alpha})A_h u_h= P_h f(u_h(t)),\quad t\in (0,T], \quad u_h(0)=P_hu_0.
\end{equation}
Following the analysis in the previous section, we represent the  solution of \eqref{semi-2} as  
\begin{equation}\label{form-1d}
u_h(t)=E_h(t)P_hu_0+\int_0^t { E}_h(t-s)P_hf(u_h(s))\,ds,
\end{equation}
where $E_h(t):V_h\to V_h$ is defined by
$$
 E_h(t) = \frac{1}{2\pi i}\int_{\Gamma_{\theta,\delta}}e^{zt} \frac{g(z)}{z}\left(g(z)I+A_h\right)^{-1} \,dz.
$$

In order to bound the  FE error $e_h(t):=u_h(t)-u(t)$, we  introduce the operator 
 $$S_h(z):=(g(z) I+A_h)^{-1}P_h-(g(z) I+A)^{-1},$$ 
which   satisfies the following properties, see \cite{LST-1996}.
\begin{lemma}\label{G} 
The following estimate holds for all $z\in\Sigma_\theta$,
\begin{equation*} 
\|S_h(z)v\|+h\|\nabla S_h(z)v\|\leq ch^2 \|v\|,
\end{equation*}
where $c$ is independent of $h$.
\end{lemma}
Let $F_h(t)=E_h(t)P_h-E(t)$. Then, by Lemma \ref{G}, $F_h(t)$ satisfies
\begin{equation} \label{0-p}
\|F_h(t)v\|+h\|\nabla F_h(t)v\| \leq ct^{-(1-\alpha)(1-\nu/2)} h^2 \|v\|_{\dot H^\nu(\Omega)},\quad 
{  \nu \in  [0,2]}.
\end{equation}
Now we are ready to prove an error estimate for the semidiscrete problem \eqref{semi-2}.
\begin{theorem}\label{thm:semi} Let  $u_0\in \dot H^\nu(\Omega)$, $\nu\in [0,2]$.
Let $u$ and $u_h$ be the solutions of problems \eqref{main} and \eqref{semi-2}, respectively.  
Then
\begin{equation} \label{01-bb}
\|e_h(t)\|+h\|\nabla e_h(t)\|\leq ch^2 t^{-(1-\alpha)(1-\nu/2)}, \quad\ t\in (0,T].
 \end{equation}
\end{theorem}
\begin{proof} Set $\beta= (1-\alpha)(1-\nu/2)$. Then, from \eqref{form-1} and \eqref{form-1d}, we obtain after rearrangements 
\begin{equation*}
e_h(t)= F_h(t)u_0+\int_0^t  {E}_h(t-s) P_h [f(u_h(s))-f(u(s))]\,ds+\int_0^t {F}_h(t-s) f(u(s))\,ds.
\end{equation*}
%
Using the properties of $F_h$  in \eqref{0-p}  and the boundedness of $\|E_h(s)\|$ and $\|f(u(s))\|$, we deduce 
\begin{eqnarray*}
\|e_h(t)\|&\leq & \| {F}_h(t) u_0\|+cL \int_{0}^{t} \|e(s)\|\,ds + \int_{0}^{t } \| {F}_h(t-s) f(u(s))\|\,ds\\
&\leq & ch^2t^{-\beta} \| u_0\|_{\dot{H}^\nu(\Omega)}+cL \int_{0}^t  \|e(s)\|\,ds+ch^2\int_{0}^{t}(t-s)^{\alpha-1}ds \\
&\leq& ch^2t^{-\beta}+cL \int_{0}^t  \|e(s)\|\,ds+ch^2.
\end{eqnarray*}
An application of Lemma \ref{Gronwall} yields
$
\|e_h(t)\| \leq ch^2t^{-\beta}.
$
The $ H^1(\Omega)$-error estimate is derived analogously, which completes the proof.
\end{proof}

\section {Time discretization}\label{sec:TD}
\se
This section is devoted to the analysis of a convolution quadrature time discretization for \eqref{semi-2} generated by the backward Euler (BE) method. Let $0 = t_0 < t_1 < . . . < t_N = T$ be a uniform partition of the time interval $[0, T]$, with grid
points $t_n = n\tau$ and step size $\tau = T/N$.
 Integrating both sides of \eqref{semi-2} over $(0,t)$, we get 
$$u_h(t) -u_h(0) +(\partial_t^{-1}+\gamma\partial_t^{\alpha-1} )A _h u_h(t)=\partial_t^{-1}P_hf(u_h(t)).$$ 
The fully discrete problem is then obtained by approximating the continuous integral by the convolution quadratures $\partial_\tau^{-1} $, $ \partial_\tau^{\alpha-1}$ and $\partial_\tau^{-1} $, respectively, generated by the   BE method, see \cite{Lubich-2004,Lubich-2006}. 
The resulting  time-stepping scheme reads: with $U_h^0=P_hu_0$, find $U^n_h\in V_h$, $n = 1, 2, \ldots,N$, such that
\begin{equation} \label{fully-1}
 U^n_h -U^0_h +(\partial_\tau^{-1}+\gamma \partial_\tau^{\alpha-1} )A_h U^n_h=\partial_\tau^{-1}P_hf(U_h^{n}).
\end{equation} 
We shall investigate  a linearized version of \eqref{fully-1} defined by:
with $U_h^0=P_hu_0$, find $U^n_h$, $n = 1, 2, \ldots,N$, such that
\begin{equation} \label{fully-2}
U^n_h -U^0_h +(\partial_\tau^{-1}+\gamma \partial_\tau^{\alpha-1} )A_h U^n_h=\partial_\tau^{-1}P_hf(U_h^{n-1}).
\end{equation}
In an expanded form,  we have 
$$
U_h^n-U_h^0+\tau A_h\sum_{j=0}^n q_{n-j}^{(1)} U^n_h+\gamma\tau^{1-\alpha}A_h \sum_{j=0}^n q_{n-j}^{(1-\alpha)} U_h^j= \tau \sum_{j=1}^n q_{n-j}^{(1)} f_h(U_h^{j-1}),
$$
where $f_h=P_hf$ and $q_{j}^{(\alpha)}= (-1)^{j}
\left(\begin{array}{c}
-\alpha\\
j
\end{array}\right),$ see \cite{Lubich-2004,Lubich-2006}.
Rewriting  \eqref{fully-2}  as
\begin{equation}\label{semi-1b}
  U_h^n =(I+(\partial_\tau^{-1}+\gamma\partial_\tau^{\alpha-1}) A_h)^{-1}\left(  U_h^0  + \partial_\tau^{-1}  f_h(U_h^{n-1})\right) ,
\end{equation}
and noting that $U_h^n$ depends linearly and boundedly on $U_h^0$,  and $ f_h(U_h^{j-1})$, $1\leq j\leq n$, 
we deduce the existence of linear and bounded operators $P_n$ and $R_n:V_h\to V_h$, $n\geq 0$, such that  $U_h^n$ is represented  by
\begin{equation}\label{semi-1e}
U_h^n = P_n U_h^0 + \tau \sum_{j=1}^n R_{n-j}  f_h(U_h^{j-1}),
\end{equation}
see \cite[Section 4]{Lubich-2006}.
 The operators $\tau R_n$, $n\geq 0$,  in \eqref{semi-1e} are the convolution quadrature weights corresponding to the Laplace transform
$K(z)=z^{-1}(I+(z^{-1}+\gamma z^{\alpha-1}) A_h)^{-1}$. 
 Since $\|K(z)\|\leq c|z|^{-1}$, an application of Lemma 3.1 in \cite{Lubich-2006}, with $\mu=1$, shows that there is a constant $B>0$, independent of $\tau$, such that 
\begin{equation}\label{R_n}
\|R_n\|\leq B ,\quad n=0,1,2,\ldots.
\end{equation}

For the error analysis, we introduce the intermediate  $v_h(t)\in V_h$ satisfying 
\begin{equation} \label{vva}
\partial_t v_h+(1+\gamma\partial_t^{\alpha})A_hv_h=P_hf(u(t)),\quad v_h(0)=P_h u_0,
\end{equation}  
and the discrete solution $v_h^n\in V_h$ defined by  
\begin{equation} \label{vv}
\partial_\tau v_h^n+(1+\gamma\partial_\tau^{\alpha})A_hv_h^n=P_hf(u(t_n)),\quad n\geq 1,\quad v_h^0=U_h^0.
\end{equation}  
Then an estimation of  $u(t_n)-v_h^n$ is given in the following lemma.
\begin{lemma}  Let $v_h^n$ be the solution to problem \eqref{vv} with $u_0\in \dot{H}^\nu(\Omega)$, $\nu\in(0,2]$. Then  there holds
\begin{equation} 
\begin{split} \label{vv-1}
\|u(t_n)-v_h^n\|\leq &  ct_n^{(1-\alpha)\nu/2-1}\tau+c t_n^{-(1-\alpha)(2-\nu)/2}h^2.
\end{split} 
\end{equation}
\end{lemma}
\begin{proof} 
Note that \eqref{vva} and \eqref{vv} can be seen as  semidiscrete and full discretizations of \eqref{main} with a given right-hand side  function $f(u(t))$, respectively. For the homogeneous case $f=0$, the bound \eqref{vv-1} can be found in  \cite[Remark 4.3]{EJLZ2016}. For the inhomogeneous  case with $u_0=0$, we consider the splitting
$$
u(t_n)-v_h^n=(u(t_n)-v_h)+(v_h-v_h^n)=:I_1+I_2.
$$
Then, from the proof  of Theorem \ref{T-1}, it is easily seen that $\|I_1\|\leq  ch^2$.
To estimate $\|I_2\|$, we follow the arguments in the proof of \cite[Theorem 3.6]{JLZ2016} with $G(z)=\frac{g(z)}{z}(g(z)I+A_h)^{-1}$. Using the bound  $\|  u'(s)\|\leq cs^{(1-\alpha)\nu/2-1}$ in Theorem \ref{T-1}, we then deduce that
\begin{eqnarray*} 
\|I_2\| &\leq &  c\tau\|f(u_h(0))\|+ c\tau\int_0^{t_n}\|f'(u(s))  u'(s)\|\,ds\leq  c t_n^{(1-\alpha)\nu/2}\tau,
\end{eqnarray*} 
which completes the proof.
\end{proof}
\begin{remark}
The bound for $\|I_2\|$ does not hold when $\nu=0$, i.e., $u_0\in L^2(\Omega)$. This is due  to the strong singularity in the bound of  $\|  u'(s)\|$.
\end{remark}

Now we are ready to derive  error estimates for the   linearized time-stepping scheme  \eqref{fully-2}.
%
%
%
\begin{theorem}\label{thm:fully-2} Let $u_0\in \dot H^\nu(\Omega)$, $\nu\in (0,2]$.
Then the fully discrete scheme \eqref{fully-2} has a unique solution $U_h^n\in V_h$, $0<n\leq N$,  satisfying
\begin{equation} \label{estimate-1c}
\|U_h^n-u(t_n)\|\leq c t_n^{(1-\alpha)\nu/2-1}\tau+c t_n^{-(1-\alpha)(2-\nu)/2}h^2,\quad 0<t_n\leq T,
\end{equation} 
where the constant $c=c(\alpha,\nu,T)$  is independent of $\tau$.
\end{theorem}
\begin{proof} 
Notice that \eqref{fully-2} is essentially a linear system with a symmetric positive definite matrix. Thus, for given $U_h^0,\cdots, U_h^{n-1}$, \eqref{fully-2} has a unique solution $U_h^n\in V_h$. 
Similar to \eqref{semi-1e}, the  solution of  \eqref{vv}   may be written as
\begin{equation}\label{semi-1d}
v_h^n = P_n U_h^0  + \tau \sum_{j=0}^n R_{n-j}  f_h(u(t_{j})),\quad n\geq 1,
\end{equation}
and in view of \eqref{semi-1d} and  \eqref{semi-1e}, we have for $0<t_n\leq T$,
\begin{eqnarray*}
U_h^n-u(t_n) & = & U_h^n-v_h^n+v_h^n-u(t_n) \\
& = & v_h^n-u(t_n) +\tau \sum_{j=1}^n R_{n-j} ( f_h(U_h^{j-1})-  f_h(u(t_{j-1})))\\
& & +\tau \sum_{j=1}^n R_{n-j} (  f_h(u(t_{j-1}))- f_h(u(t_j)))-\tau R_{n}f_h(u(t_0))=:\sum_{i=1}^4 I_i.
\end{eqnarray*}
Using \eqref{vv-1}, we readily get  $\| I_1\| \leq c t_n^{(1-\alpha)\nu/2-1}\tau+c t_n^{-(1-\alpha)(2-\nu)/2}h^2.$
For the second term, we use the Lipschitz continuity of $f$ and the estimate \eqref{R_n} to obtain (after a shifting in the summation), 
\begin{eqnarray*}
\| I_2\|&\leq&  L B \tau \sum_{j=0}^{n-1} \| U_h^j-u(t_j)\|.
\end{eqnarray*} 
To bound $I_3$, we use \eqref{R_n}, the Lipschitz continuity of $f$ and the  estimate $\| u'(t)\|\leq c t^{(1-\alpha) \nu /2-1}$, so that
\begin{eqnarray*}
\|I_3\|&\leq &   \tau LB \sum_{j=1}^{n-1}   \Vert u(t_{j+1})- u(t_j)\Vert+\tau LB \Vert u(t_{1})- u(t_0)\Vert\\
        &\leq &\tau LB \sum_{j=1}^{n-1}   \tau \sup_{t_j\leq s \leq t_{j+1}}\Vert  u'(s)\Vert+ c\tau LB\\
        &\leq &\tau LB  \sum_{j=1}^{n-1} t_j^{(1-\alpha)\nu/2-1}\tau+c\tau LB\\
    &\leq &c\tau LB   t_n^{(1-\alpha)\nu/2},
\end{eqnarray*}
where  $\Vert u(t)\Vert\leq c$ is used. 
For the last term,   \eqref{R_n} and the Lipschitz continuity of $f$ implies that
$\Vert I_4\Vert \leq   cB\tau$. Altogether, we obtain
\begin{eqnarray*}
\|U_h^n-u(t_n)\|&\leq & 
c t_n^{(1-\alpha)\nu/2-1}\tau+c t_n^{-(1-\alpha)(2-\nu)/2}h^2 + \tau LB \sum_{j=0}^{n-1}  \|U_h^{j}-u(t_{j})\|.
\end{eqnarray*}
Finally, the desired estimate \eqref{estimate-1c} follows by applying the discrete Gr\"onwall inequality. 
\end{proof}

\section{The lumped mass FEM} \label{sec:LFE}
\se
In this section, we consider the lumped mass piecewise linear FE method  and  derive related convergence rates for smooth and nonsmooth initial data. We begin by defining the quadrature approximation of the $L^2(\Omega)$-inner product on $V_h$  by
\begin{equation*}
( w,\chi)_h=\sum_{K\in\mathcal{T}_h} Q_{K,h}(w\chi) \quad \mbox{with} \quad Q_{K,h}(f)=\frac{|K|}{3}\sum_{i=1}^3f(P_i)\approx\int_Kf\,dx,
\end{equation*} 
where $P_i$, $i=1,2,3$  are vertices of the triangle $K\in\mathcal{T}_h$.
Then the spatially lumped mass FE scheme for \eqref{main} reads: find $\bar{u}_h(t) \in V_h$ such that
\begin{equation} \label{Lsemi-1}
(\partial_t \bar{u}_{h} ,\chi)_h+  a( \bar{u}_h,\chi)+ \gamma a(\partial_t^{\alpha} \bar{u}_h,\chi)=  (f(\bar{u}_h),\chi)\quad
\forall \chi\in V_h,\quad t\in (0,T], \quad \bar{u}_h(0)=P_h u_0.
\end{equation}

Next we introduce the projection operator 
$\bar{P}_h:L^2(\Omega)\rightarrow V_h$  defined   by
$(\bar{P}_h v, \chi)_h=(v, \chi)$ for all $\chi \in V_h $, and the discrete operator $\bar A_h:V_h\rightarrow V_h$  corresponding to the inner product 
$(\cdot,\cdot)_h$ satisfying  
\begin{equation*} 
(\bar A_h \psi,\chi)_h=(\nabla \psi,\nabla \chi)  \quad \forall \psi,\chi\in V_h.
\end{equation*}
Then  (\ref{Lsemi-1}) is  equivalent to
\begin{equation} \label{Lsemi-2}
\partial_t \bar{u}_{h}(t)+ ( 1+ \gamma \partial_t^{\alpha})\bar{A}_h \bar{u}_h= \bar{P}_h f(\bar{u}_h(t)),\quad \bar{u}_h(0)=P_hu_0.
\end{equation}

Set $\bar{e}_h=\bar{u}_h(t)-u(t)$ and consider the splitting $\bar{e}_h=\bar{u}_h(t)-u_h(t)+u_h(t)-u(t)=:\xi(t)+e_h(t)$, \textcolor{black}{where $u_h$ is the solution of \eqref{semi-2}}.
Subtracting $\eqref{semi-1}$ from $ \eqref{Lsemi-1}$, we have $\forall \chi \in V_h$
\begin{equation*}
(\xi'(t) ,\chi)_h+  ( \nabla\xi(t),\nabla\chi)+ \gamma (\partial_t^{\alpha} \nabla\xi(t),\nabla\chi)= (u_{h}' ,\chi)-(u_{h}' ,\chi)_h+  (f(\bar{u}_h),\chi)- (f(u_h),\chi).
\end{equation*}
Hence, $\xi(t)$ satisfies
\begin{equation}\label{LMq}
 \xi'(t)+ ( 1+ \gamma \partial_t^{\alpha})\bar{A}_h \xi(t)= -\bar{A}_hQ_h u_{h}'(t)+ \bar{P}_h(f(\bar{u}_h(t))-f(u_h(t))),\quad t\in (0,T], \quad \xi(0)=0,
\end{equation}
where $Q_h:V_h\rightarrow V_h$ is the quadrature error defined by
\begin{equation}\label{m8}
(\nabla Q_h\chi,\nabla\psi)=(\chi,\psi)_h-(\chi,\psi)\quad \forall \psi \in V_h.
\end{equation}
A key property of $Q_h$ is given in the following lemma, see \cite{CLT-2012}.
\begin{lemma}\label{lem:Qh}
Let $Q_h$ be defined by \eqref{m8}. Then there holds
\begin{equation*}
\|\nabla Q_h\chi\|+h\|\bar{A}_hQ_h\chi\|\leq ch^{p+1}\|\nabla^{p}\chi\|\quad \forall \chi\in V_h, \quad p=0,1,
\end{equation*}
where the constant $c$ is independent of $h$.
\end{lemma}
Solving \eqref{LMq} for $\xi$ using the Laplace transform, we have
\begin{equation}\label{LMS}
\xi(t)=\int_0^t { \bar E}_h(t-s)\left[-\bar{A}_hQ_h u_{h}'(s) + \bar{P}_h(f(\bar{u}_h(t))-f(u_h(t)))\right]ds,
\end{equation}
where
$$
 \bar E_h(t) = \frac{1}{2\pi i}\int_{\Gamma_{\theta,\delta}}e^{zt} \frac{g(z)}{z}\left(g(z)I+\bar{A}_h\right)^{-1}dz.
$$
Since the operator $\bar{A}_h$ is selfadjoint and positive definite,  $\bar E_h(t)$ satisfies  (see Lemma \ref{LL})
\begin{equation} \label{Lbb}
 \| \bar A^{p/2}  \partial_t^m \bar E_h(t)v\|\leq ct^{-m-(1-\alpha)(p-q)/2}  \|\bar A^{q/2}  v\|.
 \end{equation}
 Error estimates for smooth initial date are given in the following theorem.
\begin{theorem} Let $u$ be the solution of \eqref{main}  with $u_0\in \dot H^\nu(\Omega)$, $\nu\in [1,2]$. Let  $\bar{u}_h$ be the solution of  \eqref{Lsemi-2}.  
Then 
\begin{equation} \label{0-Lbb}
\|\bar{e}_h(t)\|+h\|\nabla \bar{e}_h(t)\|\leq ch^2 t^{-(1-\alpha)(1-\nu/2)}, \quad t\in (0,T].
 \end{equation}
\end{theorem}
\begin{proof} Recall that $\bar{e}(t)=\xi(t)+e_h(t)$. In Theorem \ref{thm:semi},  a bound for $e_h(t)$ is given. To estimate  $\xi(t)$, we modify the arguments presented in \cite{CLT-2012} for the parabolic case. We shall  consider the cases $\nu=2$ and $\nu=1$ separately.
For $\nu =2$, we use \eqref{Lbb} with $p=2,\,  q=1$, the Lipschitz continuity of $f$ and Lemma \ref{lem:Qh} to get
\begin{eqnarray*}
\Vert \xi(t)\Vert & \leq & \int_0^t \left[\Vert \bar E_h(t-s)\bar{A}_hQ_h u_{h}'(s)\Vert + \Vert \bar E_h(t-s)\bar{P}_h(f(\bar{u}_h(s))-f(u_h(s)))\Vert\right] \,ds\\
        & \leq &c \int_0^t \left[(t-s)^{(\alpha-1)/2}\Vert \nabla Q_h u_{h}'(s)\Vert + \Vert \xi(s)\Vert\right] \,ds\\
        & \leq & c\int_0^t \left[h^2(t-s)^{(\alpha-1)/2}\Vert \nabla u_{h}'(s)\Vert + \Vert \xi(s)\Vert\right] \,ds.\\        
\end{eqnarray*}  
Note that $\Vert \nabla  u_{h}'(t)\Vert\leq ct^{-(\alpha+1)/2}$ by Theorem \ref{T-1}. Therefore 
\begin{equation*}
\Vert \xi(t)\Vert \leq  c\int_0^t \left[h^2(t-s)^{(\alpha-1)/2}s^{-(\alpha+1)/2} + \Vert \xi(s)\Vert\right] \,ds\leq  ch^2,
\end{equation*}
where the last inequality follows by applying Lemma \ref{Gronwall}. Again, using \eqref{Lbb} with $p=1,\,  q=0$, the Lipschitz continuity of $f$ and Lemma \ref{lem:Qh}, we find that
\begin{eqnarray*}
\Vert\nabla \xi(t)\Vert & \leq & \int_0^t \left[\Vert\nabla \bar E_h(t-s)\bar{A}_hQ_h u_{h}'(s)\Vert + \Vert \nabla \bar E_h(t-s)\bar{P}_h(f(\bar{u}_h(s))-f(u_h(s)))\Vert\right] \,ds\\
        & \leq &c \int_0^t \left[(t-s)^{(\alpha-1)/2}\Vert \bar{A}_h Q_h u_{h}'(s)\Vert + (t-s)^{(\alpha-1)/2}\Vert \xi(s)\Vert\right] \,ds\\
        & \leq & c\int_0^t \left[h(t-s)^{(\alpha-1)/2}\Vert \nabla u_{h}'(s)\Vert + (t-s)^{(\alpha-1)/2}\Vert \xi(s)\Vert\right] \,ds     \\
         & \leq & c\int_0^t \left[h(t-s)^{(\alpha-1)/2}s^{-(\alpha+1)/2} + (t-s)^{(\alpha-1)/2}\Vert \xi(s)\Vert\right] \,ds,   
\end{eqnarray*}
and therefore $\Vert\nabla \xi(t)\Vert \leq ch$ by  Lemma \ref{Gronwall}. Hence, we have
\begin{equation}\label{q2}
\Vert \xi(t)\Vert + h \Vert\nabla \xi(t)\Vert\leq  ch^2.
\end{equation}

For $\nu = 1$, we split the integral in \eqref{LMS} as 
$$ \int_0^t \bar E_h(t-s)\bar{A}_hQ_h u_{h}'(s)  \,ds= \left\lbrace \int_0^{t/2}+ \int_{t/2}^t\right\rbrace  \bar E_h(t-s)\bar{A}_hQ_h u_{h}'(s) \,ds =:I_1+I_2.      
$$
To bound $I_1$, we integrate by parts so that
\begin{eqnarray*}
I_1 & =& \int_0^{t/2} \bar E_h(t-s)\bar{A}_hQ_h u_{h}'(s) \,ds    \\
    &=& \bar E_h(t/2)\bar{A}_hQ_h u_{h}(t/2)-\bar E_h(t)\bar{A}_hQ_h u_{h}(0)+ \int_0^{t/2} \bar E_h^{'}(t-s)\bar{A}_hQ_h u_{h}(s) \,ds. 
\end{eqnarray*}
By \eqref{Lbb} and  Lemma \ref{lem:Qh}, it follows that
\begin{eqnarray*}
 \Vert I_1\Vert &\leq & ct^{(\alpha-1)/2}\Vert \nabla Q_h u_{h}(t/2)\Vert +c \int_0^{t/2} (t-s)^{(\alpha-3)/2}\Vert \nabla Q_h u_{h}(s)\Vert \, ds \\
        &\leq &ch^2 t^{(\alpha-1)/2}\Vert \nabla  u_{h}(t/2)\Vert +c h^2 \int_0^{t/2} (t-s)^{(\alpha-3)/2}\Vert \nabla   u_{h}(s)\Vert \, ds\\
        &\leq &ch^2 t^{(\alpha-1)/2}.  
\end{eqnarray*}
For $I_2$,  we apply \eqref{Lbb} with $p=2,\,  q=1$  and Lemma \ref{lem:Qh} to get
\begin{eqnarray*}
\Vert  I_2 \Vert & =& \Vert\  \int_{t/2}^t \bar E_h(t-s)\bar{A}_hQ_h u_{h}'(s) \,ds \Vert\ \\
    &\leq &c h^2 \int_{t/2}^t (t-s)^{(\alpha-1)/2}\Vert \nabla u_{h}'(s)\Vert  \,ds   \\
   &\leq &  c h^2 \int_{t/2}^t (t-s)^{(\alpha-1)/2}s^{-1}\, ds\\
   &\leq &  ch^2t^{(\alpha-1)/2}.
\end{eqnarray*}
From \eqref{LMS}, we thus deduce that 
\begin{eqnarray*}
\Vert \xi(t)\Vert  & \leq & ch^2 t^{(\alpha-1)/2} + c\int_0^t  \Vert \xi(s)\Vert \,ds.
\end{eqnarray*}
Then an application of Lemma \ref{Gronwall} yields
$
\Vert \xi(t)\Vert   \leq   c h^2 t^{(\alpha-1)/2}.    
$
For the $H^1(\Omega)$-estimate of $\xi$, we  follow previous arguments,   apply \eqref{Lbb} with $p=1,\,  q=0$  and use Lemma \ref{lem:Qh} to conclude that  $ \Vert \nabla\xi(t)\Vert   \leq   c h t^{(\alpha-1)/2}.  $ 
Hence,  for $\nu=1$,
\begin{equation}\label{q1}
\Vert \xi(t)\Vert  +h  \Vert \nabla\xi(t)\Vert \leq   c h^2 t^{(\alpha-1)/2}.    
\end{equation}
By interpolation of \eqref{q2} and \eqref{q1}, we obtain 
$$\Vert \xi(t)\Vert  +h\Vert \nabla\xi(t)\Vert \leq  ch^2 t^{-(1-\alpha)(1-\nu/2)},\quad \nu \in [1,2].$$
Together with the estimate \eqref{01-bb}, this completes the proof of \eqref{0-Lbb}.
\end{proof}

In the next theorem, a nonsmooth  data error estimate is derived. The proof is quite similar to the previous one and hence omitted.
\begin{theorem}\label{nonsmooth} Let $u$ be the solution of \eqref{main}  with $u_0\in L^2(\Omega)$. 
Let $\bar{u}_h$ be the solution of \eqref{Lsemi-2}.  
Then
\begin{equation} \label{01-Lbb}
\|\bar{e}_h(t)\|+h\|\nabla \bar{e}_h(t)\|\leq ch t^{(\alpha-1)} \quad t\in (0,T].
 \end{equation}
 Furthermore, if the quadrature error operator $Q_h$ satisfies 
\begin{equation}\label{sym}
\|Q_h\chi\|\leq ch^2 \|\chi\| \quad \forall \chi\in V_h,
\end{equation}
then the following optimal error estimate holds:
\begin{equation}\label{m12b}
\|\bar{e}_h(t)\| \leq ch^2t^{(\alpha-1)}. 
\end{equation}
\end{theorem}
\begin{remark}
For symmetric meshes, the operator $Q_h$ satisfies \eqref{sym}, see \cite{CLT-2012, CLT-2013}. Thus, by interpolating  \eqref{0-Lbb} and \eqref{m12b}, we get for $u_0\in \dot{H}^{\nu}(\Omega)$ and $\nu \in [0,2]$,
$$ \Vert \bar{e}_h(t)\Vert \leq ch^2 t^{-(1-\alpha)(1-\nu/2)} .$$
\end{remark}

Now we consider the lumped mass FE method combined with a time convolution quadrature  
generated by the  backward Euler method. The resulting  linearized time-stepping scheme is  defined as follows:
with $\bar{u}_h^0=P_hu_0$, find $\bar{u}^n_h\in V_h$, $n = 1, 2, \ldots,N$, such that
\begin{equation} \label{Lfully-2}
\bar{u}^n_h -\bar{u}^0_h +(\partial_\tau^{-1}+\gamma \partial_\tau^{\alpha-1} )\bar{A}_h \bar{u}^n_h=\partial_\tau^{-1}\bar{P}_hf(\bar{u}_h^{n-1}).
\end{equation}
Following the analysis in Section \ref{sec:TD},  we obtain the following error estimate.
\begin{theorem}\label{thm:Lfully-2} Let $u_0\in \dot H^\nu(\Omega)$, $\nu\in (0,2]$. Assume the mesh is symmetric. Then  the fully discrete scheme \eqref{Lfully-2} has a unique solution $\bar{u}_h^n\in V_h$, $0<n\leq N$,  satisfying
\begin{equation} \label{Lestimate-1c}
\|\bar{u}_h^n-u(t_n)\|\leq c t_n^{(1-\alpha)\nu/2-1}\tau+c t_n^{-(1-\alpha)(2-\nu)/2}h^2,\quad 0<t_n\leq T,
\end{equation} 
where the constant $c=c(\alpha,\nu,T)$  is independent of $\tau$.
\end{theorem}


\section{ Numerical Experiments}  \label{sec:NE}
\se
In this section,  numerical  examples are provided to validate the theoretical results.  
We choose $\Omega=(0,1)^2$, fix $T=1$ and consider problems with smooth and nonsmooth initial data. 
We let $N$ denote the number of time steps and $\tau=T/N$.  Since exact solutions are difficult to obtain, we compute  reference solutions on  very refined meshes.



We shall apply the linearized time-stepping scheme \eqref{Lfully-2} 
and  perform the  computation on symmetric and nonsymmetric triangular meshes. For the symmetric meshes,  we divide the domain $\Omega$ into regular right triangles with $M$ equal subintervals of length $h=1/M$ on each side of the domain. The nonsymmetric meshes are constructed by choosing $M$ subintervals  of lengths $4/3M$ and $2/3M$ in the $x$-direction, \textcolor{black}{distributed} such that they form an alternating series, while the $y$-direction is divided into $3M/4$ equally spaced subintervals with the assumption that $M$ is divisible by 4. 
\begin{table}[ht]
\begin{center} 
\caption{$L^2$-error for cases (a) and (b); spatial error with $N=500$.}
\label{table:d2-a}
\begin{tabular}{|c|c|ccccc|c|}
\hline
$\alpha$ & case$\backslash M$  &  8 & 16 & 32 & 64 & 128 & rate\\
\hline
 & (a)  & 1.03e-3    & 2.64e-4          & 6.63e-5     &  1.65e-5        & 4.06e-6        & $ 2.03$   \\ 
0.25 & (b) & 1.03e-3            & 2.62e-4      & 6.57e-5       & 1.64e-5    & 4.03e-6     &  $ 2.02$ \\ 
  
 \hline
 & (a)   & 1.10e-3     &  2.81e-4       & 7.06e-5       & 1.76e-5     &  4.32e-6     &  $ 2.03$ \\
0.5& (b) &  1.09e-3    & 2.77e-4      & 6.95e-5      & 1.73e-5     &  4.29e-6       &  $ 2.02$ \\
  
\hline
 & (a)  & 1.16e-3     &  2.97e-4    & 7.47e-5     & 1.86e-5     &  4.57e-6     & $ 2.03$  \\  
0.75& (b) &  1.16e-3 &  2.93e-4    & 7.32e-5    &  1.83e-5    & 4.54e-6    &  $ 2.01$  \\
\hline
\end{tabular}
\end{center}
\end{table}
\begin{table}[ht]
\begin{center} 
\caption{$L^2$-error for cases (a) and (b); temporal error with $h=1/512$.}
\label{table:d2-b}
\begin{tabular}{|c|c|ccccc|c|}
\hline
$\alpha$ & case$\backslash N$  &  5 & 10 & 20 & 40 & 80 & rate\\
\hline
 & (a)  & 2.72e-4     & 1.33e-4    & 6.50e-5    &  3.10e-5      &1.42e-5   & $ 1.13$  \\ 
0.25 & (b) & 3.01e-4      & 1.19e-4    &  5.21e-5    & 2.35e-5  & 1.04e-5    &  $ 1.18$  \\ 
  
 \hline
 & (a)  &5.80e-4   &  2.88e-4    &  1.41e-4    & 6.75e-5   &  3.08e-5    & $ 1.13$  \\  
0.5& (b) & 5.43e-4    & 2.28e-4    & 1.03e-4    &  4.74e-5    &  2.12e-5    &  $ 1.16$ \\
\hline
 & (a)  & 9.39e-4   & 4.75e-4     & 2.35e-4    & 1.13e-4     &  5.18e-5    & $ 1.13$  \\  
0.75& (b) & 6.44e-4   &  2.91e-4    & 1.36e-4    &  6.39e-5     &  2.89e-5     &  $ 1.15$ \\
\hline
\end{tabular}
\end{center}
\end{table}

 We consider the model \eqref{main} with the following data:
\begin{itemize}
\item[ (a)]  $u_0(x,y)=xy(1-x)(1-y)\in \dot H^{2}(\Omega)$ and $f=\sqrt{1+u^2}$,  
\item[ (b)] $u_0(x,y)=\chi_{(0,1/2]\times(0,1)}(x,y)\in \dot H^{\epsilon}(\Omega)$ for $0\le \epsilon<1/2$, 
and $f=\sqrt{1+u^2}$,
\end{itemize}
where $\chi_S$ denotes the characteristic function of the set $S$.

The numerical results  on symmetric meshes are presented in Tables \ref{table:d2-a}-\ref{table:d2-d}. In Tables \ref{table:d2-a} and \ref{table:d2-b}, we investigate the spatial and temporal convergence rates, respectively.  From the tables, we observe an $O(h^2)$ rate in space and $O(\tau)$ rate in time which agrees well with our theoretical estimates. 

Table \ref{table:d2-c} displays the space prefactor convergence rates with respect to $t$. We notice that the spatial error essentially stays unchanged in the smooth case (a), whereas it behaves like $O(t^{3(\alpha-1)/4})$ in the  nonsmooth case (b). These results confirm the estimates of Theorem \ref{thm:Lfully-2}.


\begin{table}[!h]
\begin{center} 
\caption{$L^2$-error for cases (a) and (b)  with  $\alpha=0.5$:  $t\to 0$,  $h=1/64$, $N=500$.}
\label{table:d2-c}
\begin{tabular}{|c|ccccc|c|}
\hline
 $t_N$ &  1e-3 & 1e-4 & 1e-5 & 1e-6 & 1e-7 & rate\\
\hline
 (a)  & 8.04e-6    & 1.25e-5   & 1.52e-5   & 1.63e-5   & 1.68e-5   & -0.01 $(0)$\\ 
 (b) & 1.89e-4   & 4.68e-4   & 1.12e-3  & 2.65e-3   & 6.15e-3   & -0.36 $(-0.375)$\\
 \hline
\end{tabular}
\end{center}
\end{table}
By neglecting the spatial error, fixing the step size $\tau=10$ and taking $t_N\rightarrow 0$, we examine the  time prefactor. Theorem \ref{thm:Lfully-2} indicates that the error behaves like $O(t_N^{(1-\alpha)\nu/2})$ for 
$u_0 \in\dot{H}^\nu\textcolor{black}{(\Omega)}$. The numerical results presented in Table \ref{table:d2-d} show  a convergence rate of order $O(t_N^{0.5})$ for smooth data and $O(t_N^{1/8})$ for nonsmooth data, which  confirms  our convergence theory.
\begin{table}[h!]
\begin{center} 
\caption{$L^2$-error for cases (a) and (b)  with  $\alpha=0.5$:  $t\to 0$,  $h=1/512$, $N=10$.}
\label{table:d2-d}
\begin{tabular}{|c|ccccc|c|}
\hline
 $t_N$ &  1e-3 & 1e-4 & 1e-5 & 1e-6 & 1e-7 & rate\\
\hline
 (a)  & 2.01e-4   & 8.63e-5    & 2.92e-5    & 9.43e-6    & 3.01e-6    & 0.49 $(0.50)$\\ 
 (b) & 4.16e-3  & 3.21e-3     & 2.30e-3    & 1.73e-3     & 1.30e-3  & 0.12 $(0.125)$\\ 
 
  \hline
\end{tabular}
\end{center}
\end{table}

For the case of nonsymmetric meshes, we focus on spatial errors. Theorem \ref{nonsmooth} suggests  convergence rates of order $O(h^2)$ for smooth initial data and, by interpolation, $O(h^{3/2})$ for $u_0 \in \dot{H}^{1/2}$. In Table \ref{table:nonsymmetry}, the spatial discretization errors for cases (a) and (b) are presented. The results show convergence rates of order $O(h^2)$ in both cases, which may be seen unexpected. In our case, the particular choice of initial data  could have a positive effect on the convergence rate. A similar fact was also observed in the case of the finite volume method       \cite{KP2018}.

\begin{table}[ht]
\begin{center} 
\caption{$L^2$-error for cases (a) and (b)  on nonsymmetric meshes with  $\alpha=0.5$, $N=500$.}
\label{table:nonsymmetry}
\begin{tabular}{|c|ccccc|c|}
\hline
 case$\backslash M$  &  8 & 16 & 32 & 64 & 128 & rate\\
\hline
 (a)  & 1.70e-3      & 4.40e-4   & 1.11e-4      & 2.76e-5     & 6.64e-6       & 2.05 $(2.00)$\\ 
 (b) & 1.65e-3    & 4.20e-4      & 1.05e-4     & 2.61e-5     & 6.29e-6    & 2.05 $(1.50)$\\ 
  \hline
 
\end{tabular}
\end{center}
\end{table} 

\textcolor{black}{
\section{Conclusion}
In this work, we have studied a semilinear time-fractional Rayleigh--Stokes problem involving a fractional derivative in time of Riemann-Liouville type. The nonlinear term satisfies a global Lipchitz condition. We discussed stability and provided regularity results for the exact solution. Two spatially semidiscrete schemes were investigated based on the standard Galerkin and lumped mass finite element methods, respectively. A fully discrete scheme was obtained  via a
convolution quadrature in time generated by the backward Euler method, and optimal error estimates were derived for smooth and nonsmooth initial data. Several numerical experiments were   carried out on symmetric and nonsymmetric triangular meshes to validate the theoretical results.}

	

\begin{thebibliography}{99}
\bibitem{MK-2018} M. Al-Maskari  and S. Karaa,  {\em Galerkin FEM for a time-fractional Oldroyd-B fluid problem,}   Adv. Comput. Math., 45  (2019), 1005--1029.

\bibitem{MK-2019} M. Al-Maskari and S. Karaa, 
{\em Numerical approximation of semilinear subdiffusion equations with nonsmooth initial data},
SIAM J. Numer. Anal., 57  (2019), 1524--1544.

\bibitem{MK-2020} M. Al-Maskari  and S. Karaa, {\em FEM for nonlinear subdiffusion equations with a
local Lipschitz condition}, submitted.


\bibitem{MK-2018-b}  M. Al-Maskari and S. Karaa, {\em The lumped mass FEM for a time-fractional cable equation}, Appl. Numer. Math., 132 (2018), 73–-90.








 
\bibitem{Amann} H. Amann, {\em Existence and stability of solutions for semi-linear parabolic systems  and applications
to some diffusion reaction equations}, Proc. Roy. Soc. Edinburgh Sect. A,   81 (1978),  35--47.




\bibitem{EJLZ2016}  E. Bazhlekova, B. Jin, R. Lazarov and Z. Zhou,
{\em An analysis of the Rayleigh-Stokes problem for a generalized second-grade fluid},
 Numer. Math., 131 (2016),  1--31.


\bibitem{CLT-2012} {P. Chatzipantelidis, R. D. Lazarov and V. Thom\'ee},
{\em Some error estimates for the lumped mass finite element method for a parabolic problem,} 
Math. Comp., 81 (2012), 1–-20. 

\bibitem{CLT-2013} {P. Chatzipantelidis, R. D. Lazarov and V. Thom\'ee},
{\em Some error estimates for the finite volume element method for a parabolic problem,}  
Comput. Meth. Appl. Math.,  13  (2013), 251–-279. 

\bibitem{stokes1}   C. M. Chen,  F. Liu, and V. Anh,  {\em  Numerical analysis of the Rayleigh-Stokes problem for a heated generalized second grade fluid with fractional derivatives}, App. Math. and Comp., 204 (2008), 340--351.

\bibitem{stokes2} C.  M. Chen, F. Liu and V. Anh, {\em  A Fourier method and an extrapolation technique for Stokes’ first problem for a heated generalized second grade fluid with fractional derivative,} J. Comp.  App. Math, 223 (2009), 777--789.



 
\bibitem{Lubich-2006}  E. Cuesta, C. Lubich and C. Palencia, 
{\em Convolution quadrature time discretization of fractional diffusion-wave equations,} 
Math. Comp., 75 (2006),  673--696.





























\bibitem{JLZ2016}  B. Jin, R. Lazarov and Z. Zhou,
{\em Two fully discrete schemes for fractional diffusion and diffusion-wave equations with nonsmooth data,}
SIAM J. Sci. Comput.,  38  (2016), 146--170.




\bibitem{JLZ-2018}  B. Jin, B. Li and Z. Zhou,
{\em Numerical Analysis of nonlinear subdiffusion equations},  
SIAM J. Numer. Anal.  56 (2018), no. 1, 1--23.



 



 




\bibitem{KP2018} S. Karaa and A. K. Pani, {\em Error analysis of a FVEM for fractional order evolution equations with nonsmooth initial data},  
ESAIM Math. Model. Numer. Anal., 52 (2018), 773–-801.
 

\bibitem{LLSWZ-2018} D. Li, H. Liao, W. Sun, J. Wang and J. Zhang,
{\em Analysis of $L^1$-Galerkin FEMs for time-fractional nonlinear parabolic problems},
Commun. Comput. Phys., 24 (2018), 86-103.

\bibitem{LWZ-2017} D. Li, J. Wang and J. Zhang,
{\em Unconditionally Convergent $L^1$-Galerkin FEMs for Nonlinear Time-Fractional Schr\"odinger Equations,} 
SIAM J. Sci. Comput. 39 (2017), A3067--A3088. 




\bibitem{stokes4}  Y. Lin and  W. Jiang, 
{\em Numerical method for Stokes' first problem for a heated generalized second grade fluid with fractional derivative,}     Numer.  Meth.  Part. D. E., 27 (2011), 1599--1609.



\bibitem{Lubich-2004} C. Lubich, 
{\em Convolution quadrature revisited}, 
BIT Numerical Mathematics,  44  (2004), 503--514.


\bibitem{LST-1996} C. Lubich, I. H. Sloan and V. Thom\'ee, 
{\em Nonsmooth data error estimates for approximations of an 
evolution equation with a positive-type memory term,} 
Math. Comp., 65 (1996),  1--17. 



\bibitem{Mclean2010} W. Mclean,
{\em Regularity of solutions to a time-fractional diffusion equation},
ANZIAM J.,  52 (2010), 123--138.

























\bibitem{thomee1997} V. Thom\'ee,
Galerkin finite element methods for parabolic problems, 
Springer-Verlag, Berlin, 2006.

 
\bibitem {23} C. Wu,  {\em Numerical solution for Stokes' first problem for a heated generalized second grade fluid with fractional derivative,} Appl. Numer. Math., 59   (2009), 2571--2583.

\bibitem{Zaky} M.  A.  Zaky,
{\em An improved tau method for the multi-dimensional fractional Rayleigh--Stokes problem for a heated generalized second grade fluid},  Comput. Math. Appl., 75 (2018),  2243 -- 2258.

\end{thebibliography}
\end{document}